\documentclass[10pt,reqno]{amsart}

\usepackage{hyperref}
\usepackage[usenames,dvipsnames]{color}
\usepackage{amssymb}
\usepackage{amsmath}
\usepackage{amsfonts}
\usepackage{graphicx}

\usepackage{lscape}

\newtheorem{theorem}{Theorem}[section]
\newtheorem{lemma}[theorem]{Lemma}

\theoremstyle{definition}

\numberwithin{equation}{section}
\numberwithin{table}{section}

\theoremstyle{definition}

\newcommand{\G}{\mathcal{G}}
\newcommand{\baseRing}[1]{\ensuremath{\mathbb{#1}}}
\newcommand{\C}{\baseRing{C}}

\newcommand{\Z}{\baseRing{Z}}
\newcommand{\F}{\baseRing{F}}

\renewcommand{\phi}{\varphi}

\newcommand{\diag}{\operatorname{diag}}
\newcommand{\tr}{\operatorname{tr}}

\newcommand{\norm}[1]{\|#1\|}
\newcommand{\minimatrix}[4]{\begin{pmatrix}#1&#2\\#3&#4\end{pmatrix}}
\newcommand{\megamatrix}[8]{\left(\begin{array}{cc|cc}#1&#2&&\\#3&#4&&\\ \hline &&#5&#6\\&&#7&#8\end{array} \right)}

\newcommand{\CC}{\mathcal{C}}

\allowdisplaybreaks

\begin{document}
\title[Classical Kloosterman sums]{Classical Kloosterman sums:  representation theory, magic squares,
and Ramanujan multigraphs}

\author[P.S. Fleming]{Patrick S. Fleming}	
\address{Mathematics and Computer Science Department\\
South Dakota School of Mines and Technology\\
501 East Saint Joseph Street \\ Rapid City, South Dakota 57701}
\email{Patrick.Fleming@sdsmt.edu}

\author[S.R. Garcia]{Stephan Ramon Garcia}
    \address{   Department of Mathematics\\
            Pomona College\\
            610 N.~College Ave\\
            Claremont, California\\
            91711}
    \email{Stephan.Garcia@pomona.edu}
    \urladdr{http://pages.pomona.edu/\textasciitilde sg064747}

\author[G. Karaali]{Gizem Karaali}
    \email{Gizem.Karaali@pomona.edu}
    \urladdr{http://pages.pomona.edu/\textasciitilde gk014747}

\thanks{This work partially funded by NSF grant DMS-0901523 (\emph{Research Experiences for Undergraduate Faculty}).
S.R.~Garcia partially funded by NSF grants DMS-0638789 and DMS-1001614. }

\begin{abstract}
	We consider a certain finite group for which Kloosterman sums appear as character values.
	This leads us to consider a concrete family of commuting hermitian matrices which have Kloosterman
	sums as eigenvalues.  These matrices satisfy a number of ``magical'' combinatorial properties
	and they encode various arithmetic properties of Kloosterman sums.
	These matrices can also be regarded as adjacency matrices for multigraphs which display
	Ramanujan-like behavior.
\end{abstract}

\maketitle

\bibliographystyle{plain}

\section{Introduction}
	For a fixed odd prime $p$, let $\zeta = \exp(2\pi i /p)$ and define
	the classical \emph{Kloosterman sum}
	$K(a,b) := K(a,b,p)$ by setting
	\begin{equation}\label{eq-KloostermanDefinition}
		K(a,b) = \sum_{n =1}^{p-1} \zeta^{an+b\overline{n} }
	\end{equation}
	where $\overline{n}$ denotes the inverse of $n$ modulo $p$.  
	From \eqref{eq-KloostermanDefinition}, it follows that 
	$K(a,b)$ is real and that its value depends only upon the residue classes of $a$ and $b$ modulo $p$.
	In light of the fact that $K(a,b) = K(1,ab)$ whenever $p \nmid a$, we focus our
	attention mostly on Kloosterman sums of the form $K(1,u)$. 
	Moreover, we adopt the shorthand $K(u) := K(1,u)$ or even $K_u := K(1,u)$ when space is at a premium.

	In the years since they appeared in Kloosterman's paper on quadratic forms \cite{Kloosterman}, these exponential sums and
	their generalizations have found many diverse applications.  We do not attempt to give a historical
	account of the subject and instead direct the reader to \cite{DRHB,Hurt,Iwaniec, KatzBook}.
		
	In this note we construct a certain finite group for which 
	Kloosterman sums appear as character values (Section \ref{SectionG}).
	This eventually leads us to consider a concrete family of commuting hermitian matrices which have Kloosterman
	sums as eigenvalues (Section \ref{SectionConstruction}).  
	These matrices satisfy a number of ``magical'' combinatorial properties (Section \ref{SectionSubmatrices})
	and they encode various arithmetic properties of Kloosterman sums (Section \ref{SectionApplications}).  
	Moreover, these matrices can be regarded as the adjacency matrices for multigraphs which display
	Ramanujan-like behavior (Section \ref{SectionRamanujan}).	

\subsection*{Acknowledgments}  We thank Philip C. Kutzko for suggesting the initial representation
	theory project that spurred this work.  In particular, our basic approach stems from his paper \cite{Ku75}.  We also
	thank the American Institute of Mathematics (AIM) for hosting us for a week as part of the NSF-funded (DMS-0901523)
	\emph{Research Experiences for Undergraduate Faculty} program.

\section{The group ${\bf G}$ and its representation theory}\label{SectionG}
	Let $p>3$ be an odd prime and define the subgroup
	\begin{equation*}
		{\bf G} = \left \{ \left. \minimatrix{x}{y}{0}{1} \oplus \minimatrix{x^{-1} }{z}{0}{1}  \right| 
		x \in (\Z / p \Z)^{\times},\, y,z \in \Z / p \Z \right \}
	\end{equation*}
	of $GL_4(\Z/p\Z)$.  Here we identify direct sums of two $2 \times 2$ matrices with the 
	corresponding $4 \times 4$ matrices.  In the following, we denote matrix groups by bold 
	capital letters (e.g., ${\bf G}$) and their elements by capital letters (e.g., $I$ denotes the 
	$4\times 4$ identity matrix in ${\bf G}$).  Elements of $\Z / p \Z$ are 
	represented by lower-case letters.

	Letting
	\begin{align}
		{\bf N} &= \left\{ \left. \minimatrix{1}{y}{0}{1} \oplus \minimatrix{1}{z}{0}{1} \right| y,z \in \Z / p \Z \right\}, \\[5pt]
		{\bf T} &= \left\{ \left. \minimatrix{x}{0}{0}{1} \oplus \minimatrix{x^{-1}}{0}{0}{1} \right| x \in (\Z / p \Z)^{\times} \right\}, 
	\end{align}
	we find that ${\bf G} = {\bf N} {\bf T}$ and ${\bf N} \cap {\bf T} = \{ I \}$.  
	Since $|{\bf T} | = p-1$ and $| {\bf N} | =  p^2$, we have
	\begin{equation*}
		| {\bf G} | = (p-1)p^2.
	\end{equation*}
The conjugacy classes of ${\bf G}$ are easily computable and are given in Table \ref{TableConjugacy}.

\begin{table}
	\begin{equation*}\footnotesize
		\boxed{
		\begin{array}{crclc} 
		\underset{\text{$p-1$ classes}}{\textsc{Type 1:}}\quad
		&\CC_1 &=& \left\{  \minimatrix{1}{y}{0}{1} \oplus \minimatrix{1}{y^{-1}}{0}{1} : y \in (\Z/p\Z)^{\times} \right\} 
			& \text{($p-1$ elements)} \\ \\
		&\CC_2 &=& \left\{  \minimatrix{1}{2y}{0}{1} \oplus \minimatrix{1}{y^{-1}}{0}{1} : y \in (\Z/p\Z)^{\times} \right\} 
			& \text{($p-1$ elements)} \\ \\
		& \vdots &&\qquad\qquad\vdots&\vdots \\
		&\CC_{p-1} &=& \left\{  \minimatrix{1}{(p-1)y}{0}{1} \oplus \minimatrix{1}{y^{-1}}{0}{1} : y \in (\Z/p\Z)^{\times} \right\} 
			& \text{($p-1$ elements)} \\ \\
		\underset{\text{$2$ classes}}{\textsc{Type 2:}}\quad
		&\CC_{p} &=& \left\{  \minimatrix{1}{y}{0}{1} \oplus \minimatrix{1}{0}{0}{1} : y \in (\Z/p\Z)^{\times} \right\} 
			& \text{($p-1$ elements)} \\ \\
		&\CC_{p+1} &=& \left\{  \minimatrix{1}{0}{0}{1} \oplus \minimatrix{1}{y}{0}{1} : y \in (\Z/p\Z)^{\times} \right\} 
			& \text{($p-1$ elements)} \\ \\
		\underset{\text{$1$ class}}{\textsc{Type 3:}}\quad
		&\CC_{p+2} &=& \left \{  \minimatrix{1}{0}{0}{1} \oplus \minimatrix{1}{0}{0}{1} \right\} 
			&\text{($1$ element)}\\ \\
		\underset{\text{$p-2$ classes}}{\textsc{Type 4:}}\quad
		&\CC_{p+3} &=& \left\{  \minimatrix{g}{y}{0}{1} \oplus \minimatrix{g^{-1}}{z}{0}{1} : y,z \in \Z/p\Z \right\} 
			& \text{($p^2$ elements)} \\ \\
		&\CC_{p+4} &=& \left\{  \minimatrix{g^2}{y}{0}{1} \oplus \minimatrix{g^{-2}}{z}{0}{1} : y,z \in \Z/p\Z \right\} 
			& \text{($p^2$ elements)} \\ \\
		& \vdots &&\qquad\qquad\vdots&\vdots \\
		&\CC_{2p} &=& \left\{  \minimatrix{g^{p-2}}{y}{0}{1} \oplus \minimatrix{g^{-(p-2)}}{z}{0}{1} : y,z \in \Z/p\Z \right\} 
			& \text{($p^2$ elements)} \\ \\
		\end{array}
		}
	\end{equation*}
	\caption{\footnotesize The conjugacy classes of ${\bf G}$ (here $g$ denotes a primitive root modulo $p$).
	In particular, ${\bf G}$ has a total of $2p$ conjugacy classes whence there exist precisely $2p$ distinct irreducible
	representations of ${\bf G}$ \cite[Theorem 27.22]{CR62}.}
	\label{TableConjugacy}
\end{table}

	Since the commutator subgroup $[{ \bf G}, {\bf G}] = {\bf N}$ of ${\bf G}$ must belong to the kernel
	of any one-dimensional representation $\pi:{\bf G} \to \C$, it follows that $\pi(NT) = \pi(N)\pi(T) = \pi(T)$ 
	for all $N \in {\bf N}$ and $T \in {\bf T}$.  Thus the one-dimensional representations 
	of ${\bf G}$ correspond to one-dimensional representations of ${\bf T} \cong ( \Z / p \Z)^{\times}$.
	Fix a primitive $(p-1)$st root of unity $\xi$.  If $T$ denotes a generator of ${\bf T}$, then 
	for $n=1,2,\ldots,p-1$ the formula
	\begin{equation*}
		\pi(T) = \xi^n, \qquad \pi(N) = 1,\quad N \in {\bf N}
	\end{equation*}
	yields $p-1$ distinct irreducible representation of ${\bf G}$.
	
	Let us now identify the remaining irreducible representations.
	As before, we let $\zeta = \exp(2\pi i /p)$.  Fixing $a,b \in \Z / p \Z$, at least one of which is nonzero,
	we claim that the map $\pi : {\bf G} \to \operatorname{End}( \C[ (\Z / p \Z)^{\times}] )$ defined by
	\begin{equation}\label{eq-BigPi}
		\pi\left( \minimatrix{x}{y}{0}{1} \oplus \minimatrix{x^{-1} }{z}{0}{1} \right) \delta_h
		=  \zeta^{ a z (xh) + b y(xh)^{-1} } \delta_{xh}
	\end{equation}
	for $h \in (\Z / p \Z)^{\times}$ is an irreducible representation of ${\bf G}$.
	Verifying that $\pi$ is a homomorphism is straightforward, so we only prove irreducibility.

	Suppose that $a \neq 0$ and note that setting
	$x = 1$, $y = 0$, and $z = 1$ in \eqref{eq-BigPi} yields
	\begin{equation*}
		\pi \left( \minimatrix{1}{0}{0}{1} \oplus \minimatrix{1}{1}{0}{1} \right) \delta_h = \zeta^{ah} \delta_h
	\end{equation*}
	for $h \in (\Z / p \Z)^{\times}$.  The preceding is just another way of saying that 
	\begin{equation}\label{eq-DiagonalDistinct}
		\pi \left( \minimatrix{1}{0}{0}{1} \oplus \minimatrix{1}{1}{0}{1} \right) =
			\diag((\zeta^a)^1,(\zeta^a)^2,\ldots,(\zeta^a)^{p-1})
	\end{equation}
	with respect to the standard basis $\{\delta_1, \delta_2, \ldots, \delta_{p-1}\}$ of $\C[(\Z / p \Z)^{\times}]$.  
	Since $a \neq 0$ it follows that $\zeta^a$ is a primitive $p$th root of unity and hence the diagonal entries of 
	\eqref{eq-DiagonalDistinct} are distinct.  Thus the subspaces of $\C[( \Z / p \Z)^{\times}]$ which 
	are invariant under the matrix \eqref{eq-DiagonalDistinct} are precisely those of the form 
	$\operatorname{span}(K)$ for some $K \subseteq (\Z / p \Z)^{\times}$.  
	Suppose that $K \neq \varnothing$, $K \neq (\Z / p \Z)^{\times}$, and $x \notin K$.  For each $k \in K$, 
	\eqref{eq-BigPi} implies that
	\begin{equation*}
		\pi\left( \minimatrix{xk^{-1}}{0}{0}{1} \oplus \minimatrix{x^{-1}k}{0}{0}{1} \right) \delta_k 
		= \delta_{ (xk^{-1})k } = \delta_x \notin \operatorname{span}(K).
	\end{equation*}
	Thus $\operatorname{span}(K)$ is not invariant under $\pi$ whence $\pi$ is irreducible, as claimed.
	The proof in the case $b \neq 0$ is similar.

	The choices $a = 1$, $b = 0$ and $a = 0$, $b = 1$ lead us to two special characters,
	whose values on the various conjugacy classes
	can be found via a geometric series argument.
	We are interested primarily in $\pi$ arising when $a \neq 0$ and $b \neq 0$.
	Let $\chi_j$ denote the trace of $\pi$ corresponding to $j = a^{-1}b$. Using the 
	transformation rules for Kloosterman sums we find that
	$\chi_j( \CC_k) = K(a,bk) = K(1,jk) = K_{jk}$ for $1 \leq j,k \leq p- 1$.
	Since the Kloosterman sums 
	$K(1), K(2),\ldots, K(p-1)$ are distinct \cite[Prop.~1.3]{Fisher},
	it follows that the characters $\chi_1,\chi_2,\ldots, \chi_{p-1}$ are distinct.
	Now we can complete the character table for ${\bf G}$ (Table \ref{TableCharacter}).

\begin{landscape}
\small
\begin{table}
\begin{equation*}
\begin{array}{|c||c|c|c|c||c|c||c||c|c|c|c|}
\hline
{\bf G}&  \CC_1  & \CC_2  & \cdots  & \CC_{p-1}  &  \CC_{p} & \CC_{p+1} & \CC_{p+2} & \CC_{p+3} & \CC_{p+4} & \cdots & \CC_{2p} \\[2pt]
\hline
& \minimatrix{1}{1}{0}{1} & \minimatrix{1}{2}{0}{1} & \cdots & \minimatrix{ 1 }{f}{0}{1} & \minimatrix{1}{1}{0}{1} & \minimatrix{1}{0}{0}{1} & \minimatrix{1}{0}{0}{1} & \minimatrix{ g}{0}{0}{1}& \minimatrix{g^2}{0}{0}{1}& \cdots &\minimatrix{g^{p-2}}{0}{0}{1} \\[2pt]
& \oplus &  \oplus &  \cdots &  \oplus &  \oplus &  \oplus &  \oplus &  \oplus &  \oplus &  \cdots &  \oplus \\
& \minimatrix{1}{1}{0}{1} &  \minimatrix{1}{1}{0}{1} &  \cdots &  \minimatrix{1}{1}{0}{1} &  \minimatrix{1}{0}{0}{1} &  \minimatrix{1 }{1}{0}{1} &  \minimatrix{1}{0}{0}{1} &  \minimatrix{g^{-1}}{0}{0}{1} &  \minimatrix{ g^{-2}}{0}{0}{1} &  \cdots &  \minimatrix{g^{-(p-2)}}{0}{0}{1} \\[2pt]
\hline
|\CC_i|& f & f & \cdots & f &f&  f&  1 & p^2 & p^2 & \cdots & p^2 \\[2pt]
\hline\hline
\chi_1 &K_1 &K_2 &\cdots&K_f &-1&-1&f& 0&0&\cdots&0\\[2pt]
\hline
\chi_2 &K_2 &K_4 &\cdots &K_{2f} &-1&-1&f&0 &0&\cdots&0\\[2pt]
\hline
\vdots & \vdots& \vdots&\ddots& \vdots& \vdots& \vdots& \vdots&\vdots& \vdots& \ddots&\vdots\\[2pt]
\hline
\chi_{p-1} &K_f &K_{2f} &\cdots&K_{f^2} &-1&-1&f&0 &0&\cdots&0\\[2pt]
\hline\hline
\chi_p & -1& -1& \cdots& -1&f&-1&f&0&0& \cdots&0\\[2pt]
\hline
\chi_{p+1} &-1&-1&\cdots&-1&-1&f&f&0&0& \cdots&0\\[2pt]
\hline\hline
\chi_{p+2} &1&1&\cdots&1&1&1&1&1&1& \cdots&1\\[2pt]
\hline\hline
\chi_{p+3} &1&1&\cdots&1&1&1&1&\xi&\xi^2& \cdots&\xi^{p-2}\\[2pt]
\hline
\chi_{p+4} &1&1&\cdots&1&1&1&1&\xi^2&\xi^4& \cdots&\xi^{2(p-2)}\\[2pt]
\hline
\vdots & \vdots& \vdots&\ddots& \vdots& \vdots& \vdots& \vdots&\vdots& \vdots& \ddots&\vdots\\[2pt]
\hline
\chi_{2p} &1&1&\cdots&1&1&1&1&\xi^{(p-2)}&\xi^{2(p-2)}& \cdots&\xi^{(p-2)^2}\\[2pt]
\hline
\end{array}
\end{equation*}

\caption{\footnotesize The character table of ${\bf G}$.  Here $\xi$ is a fixed primitive $(p-1)$st root of unity and
$f=p-1$.  Since $K_0 = K_p =  -1$, we may regard the initial string of 
$-1$'s in the row corresponding to $\chi_p$ as being $K_0$'s.  This convention will simplify
several formulas later on. }
\label{TableCharacter}
\end{table}
\end{landscape}

\section{The main construction}\label{SectionConstruction}

\subsection{The crucial lemma}\label{SubsectionLemma}

	From the representation-theoretic information computed in Section \ref{SectionG}, we will
	construct a family of commuting hermitian matrices which encode many 
	fundamental properties of classical Kloosterman sums.  Our primary tool is the following
	lemma, which is a modification of \cite[Lem.~4]{Ku75} (although there 
	the reader is simply referred to \cite[Section 33]{CR62} to compose a proof of 
	this lemma on their own).  We provide a detailed proof for the sake of completeness.

\begin{lemma}\label{LemmaPowerful}
	Let ${\bf G}$ be a finite group having conjugacy classes $\CC_1, \CC_2, \ldots, \CC_s$ 
	and irreducible representations $\pi_1, \pi_2, \ldots, \pi_s$ with corresponding characters 
	$\chi_1, \chi_2, \ldots, \chi_s$.  For $1 \leq k \leq s$, fix $z=z(k) \in \CC_k$ and let 
	$c_{i,j,k}$ denote the number of solutions $(x_i,y_j) \in \CC_i \times \CC_j$ of $xy = z$ and then let   
	$M_i = (c_{i,j,k})_{j,k=1}^s$ for $1 \leq i \leq s$.

	If $W = (w_{j,k})_{j,k=1}^s$ denotes the $s \times s$ matrix with entries
	\begin{equation}\label{eq-OmegaDefinition}
		w_{j,k} = \frac{|\CC_j|\chi_k(\CC_j)}{\dim \pi_k}, 
	\end{equation}
	and $D_i = \operatorname{diag}(w_{i,1}, w_{i,2}, \ldots, w_{i,s})$, then $W$ is invertible and
	\begin{equation}\label{eq-Similarity}
		M_i W = W D_i  
	\end{equation}
	for $i=1,2,\ldots,s$.  Moreover, if we let
	$Q = \operatorname{diag}( \sqrt{ | \CC_1| }, \sqrt{ |\CC_2| }, \ldots, \sqrt{ |\CC_s| } )$, then
	the matrices $T_i = Q^{-1}M_i Q$ are simultaneously unitarily diagonalizable.  
	To be more specific, we have $T_i U = UD_i$ for $i=1,2,\ldots,s$ where
	\begin{equation}\label{eq-UDefinition}
		U = \frac{1}{ \sqrt{ | {\bf G} |} }\left(  \textstyle\sqrt{ |\CC_j| } \chi_k(\CC_j) \right)_{j,k=1}^s
	\end{equation}
	is a unitary matrix.
\end{lemma}

\begin{proof}
	For $j=1,2,\ldots,s$ define 
	\begin{equation*}
		\mathsf{C}_j = \sum_{x \in\, \CC_j} x
	\end{equation*}
	and observe that 
	\begin{equation}\label{eq-cijk}
		\mathsf{C}_i \mathsf{C}_j = \sum_{k=1}^s c_{i,j,k} \mathsf{C}_k
	\end{equation}
	holds for $1 \leq i,j,k \leq s$.  Upon applying $\chi_k$ to $\mathsf{C}_j$ we also note that
	\begin{equation}\label{Eq:FirstTraceFormula}
		\chi_k(\mathsf{C}_j)  = |\CC_j| \chi_k(\CC_j)
	\end{equation}
	since the class function $\chi_k$ assumes the constant value $\chi_k(\CC_j)$ on $\CC_j$.

	Since each $\mathsf{C}_j$ belongs to the center $Z(\C[{\bf G}])$ of $\C[{\bf G}]$ 
	(in fact $\{\mathsf{C}_1,\mathsf{C}_2,\ldots,\mathsf{C}_s\}$ is a basis for $Z(\C[{\bf G}])$ 
	by \cite[Thm.~27.24]{CR62} or \cite[Thm.~2.4]{Isaacs}) 
	and each $\pi_k$ is irreducible, it follows that $\pi_k(\mathsf{C}_j)$ is scalar for $1 \leq j,k \leq s$ 
	(this follows from a standard version of Schur's Lemma \cite[Thm.~29.13]{CR62}). 
	Thus there exist constants $w_{jk}$ such that
	\begin{equation}\label{eq-PiIdentityMatrix}
		\pi_k(\mathsf{C}_j) = w_{j,k} I_{d_k}
	\end{equation}
	for $1\leq j,k\leq s$ where $d_k = \dim \pi_k$ and $I_{d_k}$ denotes the $d_k \times d_k$ identity matrix.  
	Taking the trace of the preceding yields 
	\begin{equation}
		\label{Eq:SecondTraceFormula} 
		\chi_k(\mathsf{C}_j) = d_k  w_{j,k}. 
	\end{equation}
	Comparing \eqref{Eq:FirstTraceFormula} and \eqref{Eq:SecondTraceFormula} we find that
	\begin{equation*}
		 |\CC_j|  \chi_k(\CC_j) =  d_k  w_{j,k},
	\end{equation*}
	which gives us the formula \eqref{eq-OmegaDefinition}.
	Applying $\pi_r$ to \eqref{eq-cijk} and using \eqref{eq-PiIdentityMatrix} we obtain
	\begin{equation*}
		w_{i,r}  I_{d_r} w_{j,r} I_{d_r} =\sum_{k=1}^s c_{i,j,k} w_{kr}I_{d_r},
	\end{equation*}
	which clearly implies that
	\begin{equation*}
		w_{i,r} w_{j,r} =\sum_{k=1}^s c_{i,j,k} w_{k,r}.
	\end{equation*}
	Now simply observe that the preceding is the $(j,r)$th entry of the matrix equation \eqref{eq-Similarity}.
	Next we note that $W = \sqrt{|{\bf G}|} QUR$ where 
	$R = \operatorname{diag}( d_1^{-1}, d_2^{-1}, \ldots, d_s^{-1} )$.  In particular, it follows that
	\begin{align*}
		QUD_iR
		&= QURD_i &&\text{($R,D_i$ are diagonal)} \\ 
		&= M_i QUR &&\text{(by \eqref{eq-Similarity})}
	\end{align*}
	whence $QUD_i = M_i QU$ since $R$ is invertible.  Since $Q$ is invertible this yields
	$T_i U = UD_i$ where $T_i = Q^{-1} M_i Q$.  The fact that $|{\bf G}|^{-1/2} U$ is unitary
	(whence $W$ is invertible) follows from the orthogonality
	of the irreducible characters $\chi_1,\chi_2,\ldots,\chi_s$.
\end{proof}

\subsection{Main construction}\label{SubsectionMain}

	We now apply Lemma \ref{LemmaPowerful} to the group ${\bf G}$ constructed
	in Section \ref{SectionG}.
	As we shall see in Section \ref{SectionApplications}, the matrices produced encode many of 
	the basic properties of Kloosterman sums.

	Recall that the $(j,k)$ entry $(M_i)_{j,k}$ of $M_i$
	is defined to be the integer $c_{i,j,k}$ described in Lemma \ref{LemmaPowerful}.
	Since ${\bf G}$ has four distinct types of conjugacy classes (see Table \ref{TableConjugacy}),
	we partition each $M_i$ into 16 submatrices.  As in Table \ref{TableCharacter} we adopt the
	convention that $f = p-1$.
	It turns out that for $1 \leq i \leq f$ each of the $M_i$ has basically
	the same structure as $M_1$, so we only display $M_1$ explicitly:
	\begin{equation}\label{eq-M1}
		M_1 = \small
		\left(
		\begin{array}{cccc|cc|c|cccc}
			c_{1,1,1}&c_{1,1,2}&\cdots&c_{1,1,f}&0&0&f&0&0&\cdots&0\\
			c_{1,2,1}&c_{1,2,2}&\cdots&c_{1,2,f}&1&1&0&0&0&\cdots&0\\[-3pt]
			\vdots&\vdots&\ddots&\vdots&\vdots&\vdots&\vdots&\vdots&\vdots&\ddots&\vdots\\
			c_{1,f,1}&c_{1,f,2}&\cdots&c_{1,f,f}&1&1&0&0&0&\cdots&0\\
			\hline
			0&1&\cdots&1&0&1&0&0&0&\cdots &0\\
			0&1&\cdots&1&1&0&0&0&0&\cdots &0\\
			\hline
			1&0&\cdots&0&0&0&0&0&0&\cdots&0\\
			\hline
			0&0&\cdots&0&0&0&0&f&0&\cdots&0\\
			0&0&\cdots&0&0&0&0&0&f&\cdots&0\\[-3pt]
			\vdots&\vdots&\ddots&\vdots&\vdots&\vdots&\vdots&\vdots&\vdots&\ddots&\vdots\\
			0&0&\cdots&0&0&0&0&0&0&\cdots&f\\
		\end{array}
		\right).
	\end{equation}
	We are interested primarily in studying the entries $c_{i,j,k}$ for $1 \leq i,j,k\leq f$ and we
	discuss them at length in Section \ref{SectionSubmatrices}.  

	In general, $M_i$ for $2 \leq i \leq f$ differs from $M_1$ only in the upper-left $3 \times 3$ blocks.
	For instance the upper-left corner of $M_2$ looks like
	\begin{equation}\label{eq-Others}
		\left(\small
		\begin{array}{ccccc|cc|c}
			c_{2,1,1}&c_{2,1,2}&c_{2,1,3}&\cdots&c_{2,1,f}&1&1&0\\
			c_{2,2,1}&c_{2,2,2}&c_{2,2,3}&\cdots&c_{2,2,f}&0&0&f\\
			c_{2,3,1}&c_{2,3,2}&c_{2,3,3}&\cdots&c_{2,3,f}&1&1&0\\
			\vdots&\vdots&\vdots&\ddots&\vdots&\vdots&\vdots&\vdots\\
			c_{2,f,1}&c_{2,f,2}&c_{2,f,3}&\cdots&c_{2,f,f}&1&1&0\\
			\hline
			1&0&1&\cdots&1&0&1&0\\
			1&0&1&\cdots&1&1&0&0\\
			\hline
			0&1&0&\cdots&0&0&0&0
		\end{array}
		\right).
	\end{equation}
	We therefore restrict our attention mostly to the case $i = 1$ since the computations for $i = 2,3,\ldots,f$
	are almost identical.
	
	Examining the character table (Table \ref{TableCharacter})
	of ${\bf G}$ tells us that
	\begin{align}
		D_i &= \diag( K_{1i},K_{2i},\ldots,K_{(p-1)i},-1,-1,\underbrace{f,f,\ldots,f}_{\text{$p-1$ times}}) , \label{eq-BigD}\\
		Q &= \diag( \underbrace{ \sqrt{f}, \sqrt{f}, \ldots, \sqrt{f} }_{\text{$p+1$ times}},1,\underbrace{p,p,\ldots,p}_{\text{$p-2$ times}}),
		\nonumber
	\end{align}
	and 
	\begin{equation*}
		W = \tiny
		\left(
		\begin{array}{cccc|cc|c|cccc}
			K_1&K_2&\cdots&K_f&-1&-1&f&f&f&\cdots&f\\
			K_2&K_4&\cdots&K_{2f}&-1&-1&f&f&f&\cdots&f\\
			\vdots&\vdots&\ddots&\vdots&\vdots&\vdots&\vdots&\vdots&\vdots&\ddots&\vdots\\
			K_f&K_{2f}&\cdots&K_{f^2}&-1&-1&f&f&f&\cdots&f\\
			\hline
			-1&-1&\cdots&-1&f&-1&f&f&f&\cdots&f\\
			-1&-1&\cdots&-1&-1&f&f&f&f&\cdots&f\\
			\hline
			1&1&\cdots &1&1&1&1&1 &1&\cdots&1\\
			\hline
			0&0&\cdots&0&0&0&p^2&p^2\xi&p^2\xi^2&\cdots&p^2\xi^{p-2}\\
			0&0&\cdots&0&0&0&p^2&p^2\xi^2&p^2\xi^4&\cdots&p^2\xi^{2(p-2)}\\
			\vdots&\vdots&\ddots&\vdots&\vdots&\vdots&\vdots&\vdots&\vdots&\ddots&\vdots\\
			0&0&\cdots&0&0&0&p^2&p^2\xi^{p-2}&p^2\xi^{2(p-2)}&\cdots&p^2\xi^{(p-2)^2}\\
		\end{array}
		\right).
	\end{equation*}
	Since Lemma \ref{LemmaPowerful} guarantees that the matrices $T_i := Q^{-1}M_iQ$ are simultaneously
	unitarily similar to the corresponding $D_i$'s, each of which has only real entries, it follows from
	the Spectral Theorem that each $T_i$ is hermitian.  For instance,
	\begin{equation*}\footnotesize
		T_1 =  
		\left(
		\begin{array}{cccc|cc|c|cccc}
			c_{1,1,1}&c_{1,1,2}&\cdots&c_{1,1,f}&0&0&\sqrt{f}&0&0&\cdots&0\\
			c_{1,2,1}&c_{1,2,2}&\cdots&c_{1,2,f}&1&1&0&0&0&\cdots&0\\[-3pt]
			\vdots&\vdots&\ddots&\vdots&\vdots&\vdots&\vdots&\vdots&\vdots&\ddots&\vdots\\
			c_{1,f,1}&c_{1,f,2}&\cdots&c_{1,f,f}&1&1&0&0&0&\cdots&0\\
			\hline
			0&1&\cdots&1&0&1&0&0&0&\cdots &0\\
			0&1&\cdots&1&1&0&0&0&0&\cdots &0\\
			\hline
			\sqrt{f}&0&\cdots&0&0&0&0&0&0&\cdots&0\\
			\hline
			0&0&\cdots&0&0&0&0&f&0&\cdots&0\\
			0&0&\cdots&0&0&0&0&0&f&\cdots&0\\[-3pt]
			\vdots&\vdots&\ddots&\vdots&\vdots&\vdots&\vdots&\vdots&\vdots&\ddots&\vdots\\
			0&0&\cdots&0&0&0&0&0&0&\cdots&f\\
		\end{array}
		\right).
	\end{equation*}
	Moreover, the unitary matrix $U$ of Lemma \ref{LemmaPowerful} is given by
	\begin{equation*}\tiny
		U = \frac{1}{p}
		\left(
		\begin{array}{cccc|cc|c|cccc}
			K_1&K_2&\cdots&K_f&-1&-1&1&1&1&\cdots&1\\
			K_2&K_4&\cdots& K_{2f}&-1 &-1 &1 &1&1&\cdots&1\\
			\vdots&\vdots&\ddots&\vdots&\vdots&\vdots&\vdots&\vdots&\vdots&\ddots&\vdots\\
			K_f& K_{2f}&\cdots& K_{f^2}&-1&-1 &1 &1 &1 &\cdots&1 \\
			\hline
			-1 &-1 &\cdots&-1&f&-1 &1&1 &1 &\cdots&1 \\
			-1&-1&\cdots&-1&-1 &f &1&1 &1 &\cdots&1 \\
			\hline
			\sqrt{f}&\sqrt{f}&\cdots &\sqrt{f}&\sqrt{f}&\sqrt{f}&\frac{1}{\sqrt{f}}&\frac{1}{\sqrt{f}} &\frac{1}{\sqrt{f}}&\cdots&\frac{1}{\sqrt{f}}\\
			\hline
			0&0&\cdots&0&0&0&\frac{p}{\sqrt{f}}&\frac{p\xi}{\sqrt{f}}&\frac{p\xi^2}{\sqrt{f}}&\cdots&\frac{p\xi^{p-2}}{\sqrt{f}}\\
			0&0&\cdots&0&0&0&\frac{p}{\sqrt{f}}&\frac{p\xi^2}{\sqrt{f}}&\frac{p\xi^4}{\sqrt{f}}&\cdots&\frac{p\xi^{2(p-2)}}{\sqrt{f}}\\
			\vdots&\vdots&\ddots&\vdots&\vdots&\vdots&\vdots&\vdots&\vdots&\ddots&\vdots\\
			0&0&\cdots&0&0&0&\frac{p}{\sqrt{f}}&\frac{p\xi^{p-2}}{\sqrt{f}}&\frac{p\xi^{2(p-2)}}{\sqrt{f}}&\cdots&\frac{p\xi^{(p-2)^2}}{\sqrt{f}}\\
		\end{array}
		\right)
	\end{equation*}
	and it has the property that $T_iU = UD_i$.
	In particular, the $k$th column of $U$ is an eigenvector of $T_i$
	corresponding to the $k$th diagonal entry of $D_i$.  
	
	In light of the block upper-triangular structure
	of $U$ and the block of zeros in the upper-right of $T_i$, it follows that the equation
	$T_iU = UD_i$ still holds if we truncate all matrices involved to their upper left
	$(p+2) \times (p+2)$ blocks.  We do so in order to remove entries that
	are irrelevant for our purposes and contain no useful information about Kloosterman sums.
	Performing this truncation we now consider instead the $(p+2) \times (p+2)$ matrices
	\begin{equation}\label{eq-TruncatedD}
			\boxed{ D_i = \diag(K_{1i},K_{2i},\ldots,K_{(p-1)i},-1,-1,p-1)  }
	\end{equation}
	and
	\begin{equation}\label{eq-TruncatedT}
		\boxed{
		T_i =  
		\left(
		\begin{array}{cccc|cc|c}
			c_{i,1,1}&c_{i,1,2}&\cdots&c_{i,1,f}&0&0&\sqrt{f}\\
			c_{i,2,1}&c_{i,2,2}&\cdots&c_{i,2,f}&1&1&0\\
			\vdots&\vdots&\ddots&\vdots&\vdots&\vdots&\vdots\\
			c_{i,f,1}&c_{i,f,2}&\cdots&c_{i,f,f}&1&1&0\\
			\hline
			0&1&\cdots&1&0&1&0\\
			0&1&\cdots&1&1&0&0\\
			\hline
			\sqrt{f}&0&\cdots &0&0&0&0\\
		\end{array}
		\right).
		}
	\end{equation}
	Unfortunately, the new $U$ obtained by truncating the original $U$ is no longer unitary.
	However, this can easily be remedied by normalizing the $(p+2)$nd column, leading
	us to redefine $U$ as follows:
	\begin{equation}\label{eq-TruncatedU}
		\boxed{
		U =
		\frac{1}{p}
		\left(
		\begin{array}{cccc|cc|c}
			K_1&K_2&\cdots&K_f&-1&-1&\sqrt{f}\\
			K_2&K_4&\cdots&K_{2f}&-1&-1&\sqrt{f}\\
			\vdots&\vdots&\ddots&\vdots&\vdots&\vdots&\vdots\\
			K_f&K_{2f}&\cdots&K_{f^2}&-1&-1&\sqrt{f}\\
			\hline
			-1&-1&\cdots&-1&f&-1&\sqrt{f}\\
			-1&-1&\cdots&-1&-1&f&\sqrt{f}\\
			\hline
			\sqrt{f}&\sqrt{f}&\cdots &\sqrt{f}&\sqrt{f}&\sqrt{f}&1\\
		\end{array}
		\right).
		}
	\end{equation}	
	In summary, the truncated matrices \eqref{eq-TruncatedD}, \eqref{eq-TruncatedT},
	and \eqref{eq-TruncatedU} satisfy $T_iU=  UD_i$ for $1 \leq i \leq p-1$.

\section{Submatrices of the $T_i$}\label{SectionSubmatrices}

For $i = 1,2,\ldots,p-1$ we let $B_i = (c_{i,j,k})_{j,k=1}^{p-1}$ denote the upper left
$(p-1) \times (p-1)$ submatrix of $T_i$ \eqref{eq-TruncatedT}.  In this section we examine the structure of these
matrices.  Some of these properties will be used in Section \ref{SectionApplications} to study Kloosterman sums
and in Section \ref{SectionRamanujan} to construct Ramanujan multigraphs.

\subsection{Computing the entries}

	We claim that the entries of the $B_i$ are given by
	\begin{equation}\label{eq-cijk}
		\boxed{ c_{i,j,k} = 1 + \left( \frac{\beta(i,j,k)}{p} \right) }
	\end{equation}
	where 
	\begin{equation}\label{eq-Beta}
		\boxed{ \beta(i,j,k) = i^2+j^2+k^2-2ij-2jk-2ik }
	\end{equation}
	and $(\frac{\cdot}{p})$ denotes the \emph{Legendre symbol}
	\begin{equation*}
		\left( \frac{a}{p} \right) = 
		\begin{cases}
			-1 & \text{if $a$ is a quadratic nonresidue modulo $p$}, \\
			0& \text{if $a \equiv 0\!\! \!\!\pmod{p}$}, \\
			1 & \text{if $a$ is a quadratic residue modulo $p$}.
		\end{cases}
	\end{equation*}
	In particular, it follows that $c_{i,j,k} \in \{0,1,2\}$ for all $1 \leq i,j,k\leq p-1$.  
	In light of \eqref{eq-cijk} and \eqref{eq-Beta}, we also have
	\begin{equation}\label{eq-Permutation}
		c_{i,j,k} = c_{\sigma(i),\sigma(j),\sigma(k)}
	\end{equation}
	for any permutation $\sigma$ of $\{i,j,k\}$ and
	\begin{equation}\label{eq-Multiply}
		c_{i,j,k} = c_{li,lj,lk}
	\end{equation}
	for $1 \leq i,j,k,l \leq p-1$ (here the subscripts $li,lj,lk$ are considered modulo $p$).
	Let us now justify the formula \eqref{eq-cijk} for the entries of $B_i$.

	According to Lemma \ref{LemmaPowerful}, the entries $c_{i,j,k}$ 
	denote the number of solutions $(X,Y) \in \CC_i \times \CC_j$ to the equation $XY =Z$
	for some fixed $Z \in \CC_k$.  For $1 \leq i,j,k \leq p-1$ we consider the equation	
	\begin{equation}\label{eq-MysteryClassEquation}
		\underbrace{\megamatrix{1}{ix}{0}{1}{1}{x^{-1}}{0}{1} }_{X \in \CC_i}
		\underbrace{ \megamatrix{1}{jy}{0}{1}{1}{y^{-1}}{0}{1} }_{Y \in \CC_j}
		=
		\underbrace{ \megamatrix{1}{k}{0}{1}{1}{1}{0}{1} }_{Z \in \CC_k},
	\end{equation}
	which instantly reveals that $x^{-1} + y^{-1} = 1$ and $ix +jy = k$.  
	Note that the first equation ensures that $x,y \neq 1$ so that $y = x(x-1)^{-1}$.  
	Substituting this into the second equation we obtain the quadratic 
	\begin{equation}\label{eq-QuadraticBad}
		i x^2 + (j-k-i)x + k =0,
	\end{equation}
	which has either $0$, $1$, or $2$ solutions in $\Z / p \Z$.  Since $k \neq 0$,
	it also follows that every solution $x$ to \eqref{eq-QuadraticBad} belongs to 
	$(\Z / p \Z)^{\times}$ and hence there is a bijective correspondence between solutions $(X,Y) \in \CC_i \times \CC_j$
	to \eqref{eq-MysteryClassEquation} and solutions $x \in \Z / p \Z$ to \eqref{eq-QuadraticBad}.
	Substituting $(2i)^{-1}[x - (j-k-i)]$ for $x$ reveals that
	\eqref{eq-QuadraticBad} has the same number of solutions as 
	\begin{equation*}
		x^2 = (j-k-i)^2 - 4ik = \beta(i,j,k)
	\end{equation*}
	where the function $\beta(i,j,k)$ is defined by \eqref{eq-Beta}.  This establishes \eqref{eq-cijk}.

	Before proceeding, we should remark that the appearance of the preceding quadratic 
	is not surprising when one considers the well-known formula 
	\begin{equation}\label{eq-Alternate}
		K(u) = \sum_{n=0}^{p-1} \left( \frac{n^2 - 4u}{p} \right) \zeta^n,
	\end{equation}
	which can be found in \cite[Lem.~1.1]{Fisher}, \cite[eq.~(1.6)]{Lehmer}, or \cite[eq.~(51)]{Salie}.

\subsection{Rows and columns}
	Fix $1 \leq i,k \leq p-1$ and note that as $X$ runs over the $p-1$ elements of $\CC_i$,
	the variable $Y = X^{-1} Z$ runs over $f$ distinct elements of ${\bf G}$.  Therefore the sum 
	of the $k$th column of $M_i$ must equal $f$.  In light of \eqref{eq-M1}, we obtain the following
	formula for the column sums of the $B_i$:
	\begin{equation}\label{eq-RowSum}
	 	\boxed{
		\sum_{j=1}^{p-1} c_{i,j,k} =
		\begin{cases}
			p-2 & \text{if $k=i$},\\
			p-3 & \text{if $k \neq i$}.
		\end{cases}
		}
	\end{equation}
	By symmetry, the same formula holds for the row sums of $B_i$.  

	For $1\leq i,j\leq p-1$ fixed we have
	\begin{equation}\label{eq-QuadraticJK}
		\beta(i,j,k)= k^2 - 2(i+j)k + (i-j)^2, 
	\end{equation}
	which we now consider as a quadratic in the variable $k$.  By \eqref{eq-cijk} it follows that
	$c_{i,j,k} = 1$ if and only if $k \in (\Z/p\Z)^{\times}$ is a root of the preceding quadratic.
	Since $i$ and $j$ are fixed, this holds for at most two values of $k$.
	Therefore each row (or column) of $B_i$ can contain at most two $1$'s. Let us be more specific.
	\begin{itemize}\addtolength{\itemsep}{0.5\baselineskip}
		\item	Since $p-2$ is odd, it follows from \eqref{eq-RowSum} that the $i$th
			row of $B_i$ contains exactly one $1$.  The remaining $p-2$ entries of
			the $i$th row are $0$'s and $2$'s which add up to $p-3$ by \eqref{eq-RowSum}.  
			Thus exactly $\frac{p-3}{2}$ of these entries are $2$'s and $\frac{p-1}{2}$ of them are $0$'s.	
		\item	Since $p-3$ is even, it follows from \eqref{eq-RowSum} that for $j \neq i$ the $j$th row
			of $B_i$ contains either zero or two $1$'s.  If the $j$th row contains zero $1$'s, then
			$\frac{p-3}{2}$ of its entries must be $2$'s.  If the $j$th row contains two $1$'s, then
			$\frac{p-5}{2}$ of its entries must be $2$'s.  
		\item Now suppose that $j\neq i$.  We claim that if $ij$ is a quadratic residue modulo $p$, then the $j$th row of $B_i$
			contains exactly two $1$'s and zero $1$'s otherwise.  The only way to obtain a $1$
			in the $j$th column of $B_i$ is for \eqref{eq-QuadraticJK} to be congruent to $0$ modulo $p$
			for some $1 \leq k \leq p-1$.  Using the substitution $k \mapsto i+j+k \pmod{p}$, we see that this occurs
			if and only if
			\begin{equation}\label{eq-NewQuadratic}
				k^2 \equiv 4 ij \pmod{p}
			\end{equation}
			for some $k$ in $\{0,1,2,\ldots,i+j-1,i+j+1,\ldots,p-1\}$.  The forbidden value $i+j$ poses no problem since if
			$(i+j)^2 \equiv 4ij \pmod{p}$, then $p|(i-j)$ whence $j = i$, contradicting our hypothesis that $j \neq i$.
			Thus \eqref{eq-NewQuadratic} has a solution $k \neq i+j$ if and only if $ij$ is a quadratic residue modulo $p$.	
	\end{itemize}

	Putting this all together, we obtain Table \ref{TableRows}, which describes the number of elements of each type
	in a given row/column of $B_i$.
	Using this data and the fact that there are exactly $\frac{p-3}{2}$ nonzero quadratic residues of the form $ij$ ($j \neq i$)
	and $\frac{p-1}{2}$ nonresidues we can compute the total number of 
	$0,1,2$'s in the matrix $B_i$ (see Table \ref{TableTotal}).  
	We can also use this information to compute the sum of the squares of the entries of 
	$B_i$ (i.e., the quantity $\tr B_i^*B_i = \tr B_i^2$):   
	\begin{align}
		\tr B_i^2 
		&= (p-2) + 2(p-2)(p-3) \nonumber\\
		&= 2p^2 - 9p + 10. \label{eq-BSOS}
	\end{align}

	\begin{table}
		\begin{equation*}
			\begin{array}{|l||c|c|c|}
			\hline
			\text{Row \#} & \text{\#0's} & \text{\#1's} & \text{\#2's} \\
			\hline
			j=i & \frac{p-1}{2} & 1 & \frac{p-3}{2} \\
			\hline
			j\neq i, (\frac{ij}{p})=1 & \frac{p-1}{2}& 2 & \frac{p-5}{2}\\
			\hline
			j \neq i, (\frac{ij}{p})=-1 & \frac{p+1}{2} & 0 & \frac{p-3}{2}\\
			\hline
			\end{array}
		\end{equation*}
		\caption{\footnotesize Number of elements of each type in a given row of $B_i$.  By symmetry, the same data applies
		to the columns of $B_i$.}
		\label{TableRows}
	\end{table}

	\begin{table}
		\begin{equation*}
			\begin{array}{|l||c|c|c|}
				\hline
				&\text{\#0's} & \text{\#1's} & \text{\#2's} \\
				\hline\hline
				\text{Total} & \frac{1}{2}p(p-1) & p-2 & \frac{1}{2}(p-2)(p-3) \\
				\hline
			\end{array}
		\end{equation*}
		\caption{\footnotesize Total number of elements of each type in the matrix $B_i$.  For large $p$
		the entries are roughly evenly split between $0$'s and $2$'s (i.e., approximately $\frac{1}{2}p^2$.  
		On the other hand, the total number of $1$'s in the matrix is only of order $p$.}
		\label{TableTotal}
	\end{table}

\subsection{Magical properties}\label{SubsectionMagic}
	Along the main diagonal of $B_i$ we have $j = k$ so that $c_{i,j,j} = 1+(\frac{i^2-4ij}{p})$.
	For $j = 1,2,\ldots,p-1$, this yields the sequence
	\begin{equation}\label{eq-ModuloSequence}
		i^2-4i,\,\, i^2 - 8i,\,\, \ldots, \,\, i^2 - 4i(p-1) \pmod{p}.
	\end{equation}
	Note that the sequence $i^2 - 4 ij = i(i - 4j)$ cannot assume the value $i^2$ since $p \nmid 4j$.
	On the other hand, $i(i - 4j)$ assumes every other value in $\Z / p\Z$ exactly once.
	Thus we conclude that $0$ appears in the sequence \eqref{eq-ModuloSequence} exactly once.  
	Therefore exactly one of the diagonal entries of $B_i$ is equal to $1$.  Since $B_i$ is symmetric, it follows that
	there are an odd number of $1$'s among its entries, in agreement with the data in Table \ref{TableTotal}.

	The trace of $B_i$ is easily computed using the above.  Since \eqref{eq-ModuloSequence}
	assumes every value in $(\Z/p\Z)^{\times}$ apart from $1$, it follows that there are precisely
	$\frac{p-3}{2}$ nonzero quadratic residues on the list.  Since we already know that a single $1$ appears
	on the diagonal of $B_i$ it follows that
	\begin{equation}\label{eq-TraceCondition}
		\tr B_i = p-2.
	\end{equation}

	Next we observe that certain ``broken diagonals'' of $B_i$ also enjoy curious summation properties.
	Indeed, using \eqref{eq-Permutation}
	and \eqref{eq-Multiply} for $j$ and $k$ fixed we have
	\begin{equation}\label{eq-Broken}
		\sum_{l=1}^{p-1} c_{i,lj,lk} 
		=  \sum_{l=1}^{p-1} c_{il^{-1},j,k}  
		= \sum_{r=1}^{p-1} c_{r,j,k} 
		= \sum_{r=1}^{p-1} c_{k,j,r} \\
		=
		\begin{cases}
			p-2 & \text{if $j = k$},  \\
			p-3 & \text{if $j \neq k$},
		\end{cases}
	\end{equation}
	by \eqref{eq-RowSum}.  
	Define an equivalence relation $\sim$ on pairs $(j,k)$ with $1 \leq j,k \leq p-1$
	by setting $(j_1,k_1) \sim (j_2,k_2)$ if and only if $j_1 = l j_2 \pmod{p}$ and $k_1 = l k_2 \pmod{p}$
	for some $1 \leq l \leq f$.  This partitions the indices $(j,k)$ into $p-1$ equivalence classes of $p-1$ elements
	each, the sum over each equivalence class being given by the preceding formula.

	Putting this all together we obtain two different ``magic matrices'' that are naturally associated to $B_i$.
	First observe from \eqref{eq-RowSum}, \eqref{eq-TraceCondition}, and \eqref{eq-Broken} that if we subtract $1$
	from $c_{i,i,i}$ we obtain a new matrix $B_i'$ for which each row, column, diagonal, and ``broken diagonal'' sums to $p-3$.
	For instance, if $p = 7$ and we subtract $1$ from the $(1,1)$ entry of $B_1$ we obtain the $6 \times 6$ ``magic'' matrix
	\marginpar{\tiny \textbf{Note}: These matrices are best viewed in color (each broken diagonal is represented using a different color).}
	\begin{equation*} 
		\left(
		\begin{array}{cccccc}
			 \color{red}{\bf 1} & \color{blue}{\bf 0} & \color{yellow}{\bf 2} & \color{cyan}{\bf 1} & \bf 0 & \color{green}{\bf 0} \\
			 \color{cyan}{\bf 0} & \color{red}{\bf 1} & \bf 0 & \color{blue}{\bf 1} & \color{green}{\bf 0} & \color{yellow}{\bf 2} \\
			\bf  2 & \color{yellow}{\bf 0} & \color{red}{\bf 0} & \color{green}{\bf 2} & \color{cyan}{\bf 0} & \color{blue}{\bf 0} \\
			 \color{blue}{\bf 1} & \color{cyan}{\bf 1} & \color{green}{\bf 2} & \color{red}{\bf 0} & \color{yellow}{\bf 0} & \bf 0 \\
			 \color{yellow}{\bf 0} & \color{green}{\bf 0} & \color{blue}{\bf 0} & \bf 0 & \color{red}{\bf 2} & \color{cyan}{\bf 2} \\
			\color{green}{\bf 0} & \bf 2 & \color{cyan}{\bf 0} & \color{yellow}{\bf 0} & \color{blue}{\bf 2} & \color{red}{\bf 0}
		\end{array}
		\right),
	\end{equation*}
	each row, column, and broken diagonal of which sums to $4$.
	
	We can also augment the $(p-1) \times (p-1)$ matrix $B_i$ with one additional row and column
	from the larger matrix $T_i$ to obtain a $p \times p$ matrix $A_i$ which also enjoys ``magic square'' properties
	(i.e., $A_i$ is the upper-left $p \times p$ principal submatrix of $T_i$).  In particular,
	each row and column of $A_i$ sums to $p-2$ while each diagonal and ``broken diagonal'' sums to $p-3$.
	For $p=11$, we obtain the $11 \times 11$ ``magical'' submatrix $A_1$ of $T_1$
	\begin{equation*}
		\left(
		\begin{array}{cccccccccc|c}
			 0 & 0 & \color{red}{\bf 0} & 1 & 2 & \color{green}{\bf 2} & 0 & 0 & 2 & \color{blue}{\bf 2} & \color{cyan}{\bf 0}\\
			\color{green}{\bf 0} & 2 & 2 & 2 & 0 & \color{red}{\bf 2} & 0 & 0 & \color{blue}{\bf 0} & 0 & \color{cyan}{\bf 1}\\
			 0 & 2 & 1 & 0 & 1 & 2 & \color{green}{\bf 0} & \color{blue}{\bf 2} & \color{red}{\bf 0} & 0 & \color{cyan}{\bf 1}\\
			 \color{red}{\bf 1} & \color{green}{\bf 2} & 0 & 0 & 0 & 0 & \color{blue}{\bf 0} & 2 & 1 & 2 & \color{cyan}{\bf 1}\\
			 2 & 0 & 1 & \color{red}{\bf 0} & 2 & \color{blue}{\bf 0} & 2 & \color{green}{\bf 0} & 1 & 0 & \color{cyan}{\bf 1}\\
			 2 & 2 & \color{green}{\bf 2} & 0 & \color{blue}{\bf 0} & 0 & \color{red}{\bf 2} & 0 & 0 & 0 & \color{cyan}{\bf 1}\\
			 0 & 0 & 0 & \color{blue}{\bf 0} & 2 & 2 & 0 & 2 & \color{green}{\bf 0} & \color{red}{\bf 2} & \color{cyan}{\bf 1}\\
			 0 & \color{red}{\bf 0} & \color{blue}{\bf 2} & \color{green}{\bf 2} & 0 & 0 & 2 & 0 & 2 & 0 & \color{cyan}{\bf 1}\\
			 2 & \color{blue}{\bf 0} & 0 & 1 & \color{red}{\bf 1} & 0 & 0 & 2 & 2 & \color{green}{\bf 0} & \color{cyan}{\bf 1}\\
			 \color{blue}{\bf 2} & 0 & 0 & 2 & \color{green}{\bf 0} & 0 & 2 & \color{red}{\bf 0} & 0 & 2 & \color{cyan}{\bf 1}\\
			 \hline
			 0 & 1 & 1 & 1 & 1 & 1 & 1 & 1 & 1 & 1 & \color{cyan}{\bf 0}
		\end{array}
		\right).
	\end{equation*}
	In particular, note that each row and column sums to $9$ while each broken diagonal sums to $8$.
	

\subsection{Qualitative behavior of eigenvalues}
	By the triangle inequality one obtains the trivial bound $| K(a,b) | \leq p-1$ for all $a,b$.  However, 
	a significant amount of cancellation can occur in the sum \eqref{eq-KloostermanDefinition}.  
	The famous \emph{Weil bound} asserts that
	\begin{equation}\label{eq-Weil}
		|K(a,b)| \leq 2 \sqrt{p},
	\end{equation}
	whenever $p \nmid ab$ \cite{Weil}.  A complete proof, based on Stepanov's method \cite{Stepanov},
	can be found in the recent text \cite[Thm.~11.11]{Iwaniec}.	

	For a $n \times n$ real symmetric matrix $X$ we let
	\begin{equation*}
		\lambda_0(X) \leq \lambda_1(X) \leq \cdots \leq \lambda_{n-1}(X)
	\end{equation*}
	denote the eigenvalues of $X$, repeated according to multiplicity.  We are concerned
	here with the qualitative behavior of the eigenvalues of the $(p+2)\times (p+2)$ matrix 
	$T :=T_1$ \eqref{eq-TruncatedT} and its $p \times p$ upper-left principal submatrix $A := A_1$.
				
	By the Weil bound \eqref{eq-Weil} and a standard result relating the eigenvalues of a hermitian matrix to those of a principal
	submatrix \cite[Thm.~4.3.15]{HJ} we have 	
	\begin{equation*}
		-2\sqrt{p} \,\leq \, \lambda_j(T) \,\leq \, \lambda_j(A)\, \leq\, \lambda_{j+2}(T) \,\leq \,2 \sqrt{p}
	\end{equation*}
	for $0 \leq j \leq p-2$.  In particular, it follows from \eqref{eq-Weil}
	and the preceding chain of inequalities that
	\begin{equation}\label{eq-Chain01}
		-2\sqrt{p} \,\leq\, \lambda_0(T) \,\leq\, \lambda_0(A) \,\leq\, \lambda_2(T) 
	\end{equation}
	and
	\begin{equation}\label{eq-Chain02}
		\lambda_{p-1}(T) \,\leq\, \lambda_{p-1}(A) \,\leq\, \lambda_{p+1}(T) \,\leq\, 2 \sqrt{p}.
	\end{equation}	
	Using the Weil bound, we now write 
	\begin{equation*}
		K(u) = 2\sqrt{p} \cos \theta_p(u)
	\end{equation*}
	where $\theta_p(u) \in [0,\pi]$.
	The \emph{vertical Sato-Tate law} \cite{Adolphson, KatzBook} states 
	that as $p \to \infty$ the sequence of angles $\theta_p(u)$ becomes equidistributed with respect
	to the \emph{Sato-Tate measure} $\mu = \frac{2}{\pi} \sin^2\theta\,d\theta$ on $[0,\pi]$.
	Thus for any fixed $\delta > 0$ there are at least three values of $\theta_p(u)$ in each of the intervals
	$[0,\delta]$ and $[\pi-\delta , \pi$] when $p$ is sufficiently large.  In light of \eqref{eq-Chain01} and \eqref{eq-Chain02},
	we see that $\lim_{p\to\infty} \lambda_0(A) = -2\sqrt{p}$ and $\lim_{p\to\infty} \lambda_{p-2}(A) = 2 \sqrt{p}$.
	This behavior is clearly reflected in Table \ref{TableEigenvalues}, even for relatively small
	values of $p$.

	\begin{table}
		\begin{equation*}\footnotesize
			\begin{array}{| r || l | l | l || l | l | l | }
			\hline
			p & -2\sqrt{p} & \lambda_0(T) & \lambda_0(A) & \lambda_{p-2}(A) 
				& \lambda_{p+1}(T) & 2 \sqrt{p} \\
			\hline\hline
			7   & -5.2915  &  -2.69202 &   -2.55594   & \bf 3.87311 &   4.49396  &  5.2915\\
			\hline
			11 &   -6.63325  &  -5.71695&   \bf  -5.3493   & 4.48588  &  4.79575  &  6.63325\\
			\hline
			29  &  -10.7703  &  -9.50028  &\bf   -9.43532  &  8.89626  &  9.06824  &  10.7703\\
			\hline
			71  &  -16.8523  &  -15.8699  & \bf  -15.8149  &  14.1059  &  14.1728  &  16.8523\\
			\hline
			113  &  -21.2603  &  -20.9713  &\bf   -20.8836  &  19.5715  &  19.6731  &  21.2603\\
			\hline
			229  &  -30.2655   & -29.8296  &  -29.75   & \bf 29.9351 &   30.0001   & 30.2655\\
			\hline
			379  &  -38.9358  &  -38.2481  &\bf   -38.2008   & 37.4232 &   37.4756  &  38.9358\\
			\hline
			541  &  -46.5188  &  -46.4712  & \bf  -46.4221   & 46.3519  &  46.3885  &  46.5188\\
			\hline
			863   & -58.7537  &  -58.5638  & \bf  -58.5258   & 57.613   & 57.6483 &   58.7537\\
			\hline
			1223  & -69.9428  & -67.6103 &  -67.5843 &\bf   69.0147 &  69.0451 &  69.9428 \\
			\hline
			1583  &  -79.5739  & -79.328 & \bf  -79.3055 &  77.3993 &  77.4206  & 79.5739 \\
			\hline
			1987  & -89.1516  & -88.7625  & -88.7417 &  \bf 88.7745 &  88.7849  & 89.1516 \\
			\hline
			\end{array}
		\end{equation*}
		\caption{\footnotesize The smallest and the second largest eigenvalues of the $(p+2)\times(p+2)$
		matrix $T$ and its $(p-1)\times(p-1)$ principal submatrix $A$.}
		\label{TableEigenvalues}
	\end{table}

	The preceding argument relies upon a deep result of Katz \cite{KatzBook}.
	On the other hand, the matrix $A$ is quite concrete and it enjoys many unusual combinatorial
	properties (Subsection \ref{SubsectionMagic}).  One might hope 
	to estimate the eigenvalues of $A$ directly to obtain an elementary proof
	of a Weil-type bound.  The following result indicates that the error incurred using such an approach would not
	change the order of magnitude of the resulting estimate.
	
	\begin{theorem}
		If $p\geq 5$ is an odd prime, then
		\begin{equation*}
			\max\{ |K(u)|:u=0,1,\ldots,p-1 \} \,\leq\,  \max\{|\lambda_0(A) |,|\lambda_{p-2}(A)| \} +  \sqrt{p-1}.
		\end{equation*}
		In other words, an estimate of the form $ \max\{|\lambda_0(A) |,|\lambda_{p-2}(A)| \} \leq C\sqrt{p}$
		leads to a Weil-type estimate of the form $\max\{ |K(u)| \} \leq C' \sqrt{p}$.
	\end{theorem}
		
	\begin{proof}
		First note that if $A_{ij}$ are block matrices of the appropriate size, then
		\begin{equation*}\footnotesize
			\left\|
			\begin{pmatrix}
				A_{11} & A_{12} & \cdots & A_{1n} \\
				A_{21} & A_{22} & \cdots & A_{2n} \\
				\vdots & \vdots & \ddots & \vdots \\
				A_{m1} & A_{m2} & \cdots & A_{mn}
			\end{pmatrix}
			\right\|
			\,\leq\,
			\left\|
			\begin{pmatrix}
				\|A_{11}\| & \|A_{12}\| & \cdots & \|A_{1n}\| \\
				\|A_{21}\| & \|A_{22}\| & \cdots & \|A_{2n}\| \\
				\vdots & \vdots & \ddots & \vdots \\
				\|A_{m1}\| & \|A_{m2}\| & \cdots & \|A_{mn}\|
			\end{pmatrix}
			\right\|.
		\end{equation*}
		Moreover, if each $A_{ij}$ is a nonnegative multiple of an all $1$'s matrix, then equality holds.  Therefore
		\begin{align*}
			\norm{ T - A\oplus 0_{3\times 3} }
			&= \left\|			\left(\footnotesize
					\begin{array}{c|ccc|c|c}
						0&0&\cdots&0&0&\sqrt{p-1}\\
						\hline
						0&0&\cdots&0&1&0\\
						\vdots&\vdots&\ddots&\vdots&\vdots&\vdots\\
						0&0&\cdots&0&1&0\\
						\hline
						0&1&\cdots&1&0&0\\
						\hline
						\sqrt{p-1}&0&\cdots&0&0&0
					\end{array}
					\right)
					\right\| \\
			&= \left\| \footnotesize
			\begin{pmatrix}
			0 & 0 & 0 & \sqrt{p-1} \\
			0 & 0 & \sqrt{p-2} & 0 \\
			0 & \sqrt{p-2} & 0 & 0\\
			\sqrt{p-1} & 0 & 0 & 0
			\end{pmatrix}
			\right\| \\
			&=\sqrt{p-1}.
		\end{align*}
		The theorem now follows from the triangle inequality for the operator norm.
	\end{proof}

\section{Applications to Kloosterman sums}\label{SectionApplications}
In the following, we employ the matrices $T = T_1$ \eqref{eq-TruncatedT},
$D= D_1$ \eqref{eq-TruncatedD}, and $U$ \eqref{eq-TruncatedU} constructed
in Subsection \ref{SubsectionMain}.  As before, we let $A=A_1$ and $B=B_1$ denote the
upper-left $p \times p$ and $(p-1) \times (p-1)$ principal submatrices of $T$.

\subsection{Basic Kloosterman identities}

	Using the character table of ${\bf G}$ (Table \ref{TableCharacter}) we can derive 
	a number of identities involving Kloosterman sums.  For instance, 
	taking the inner product of the first column with the $(p+2)$th column yields
	\begin{equation}\label{eq-KloostermanTrace}
		\boxed{ \sum_{u=0}^{p-1} K(u) =0. }
	\end{equation}	
	In particular, this gives another proof of \eqref{eq-TraceCondition} since $\tr B = \tr T = \tr D = f - 1 + \sum_{u=0}^{p-1} K(u) = p-2$.

	If $c \neq 0$, then taking the inner product of the first and the $c$th 
	of the columns of the character table leads to
	\begin{equation}\label{eq-KloostermanOrthogonal}
		\boxed{ \sum_{u=0}^{p-1} K(u) K(cu) = -p. }
	\end{equation}
	
	Taking the inner product of the first column with itself we find that
	\begin{align*}
		\sum_{u =0}^{p-1} K^2(u) + p \cdot 1^2
		&= \#C_{\bf G}\left( \minimatrix{1}{1}{0}{1} \oplus \minimatrix{1}{1}{0}{1}\right) \\
		&= \# \left\{ \minimatrix{1}{y}{0}{1} \oplus \minimatrix{1}{z}{0}{1} : y,z \in \Z / p \Z \right\}  \\
		&= p^2,
	\end{align*}
	where $C_{\bf G}(\cdot)$ denotes the centralizer of an element of ${\bf G}$.   This yields the
	well-known formula \cite[eq.~3.7]{Lehmer}:
	\begin{equation}\label{eq-KloostermanSOS}
		\boxed{ \sum_{u=0}^{p-1} K^2(u) = p^2 - p. }
	\end{equation}
	
	Since the matrices $D$ and $T$ (given by \eqref{eq-TruncatedD} and \eqref{eq-TruncatedT}, respectively)
	are unitarily similar, it follows that the sums of the squares of their entries must be equal (i.e.,
	$\tr D^2 = \tr T^2$).  Thus
	\begin{equation}\label{eq-Backward}
		\underbrace{(p^2-p) + (-1)^2 + (p-1)^2 }_{\tr D^2}
		= \underbrace{\tr B^2 + 2(p-1) + 4(p-2)+2 }_{\tr T^2}
	\end{equation}
	by \eqref{eq-KloostermanSOS}.
	The preceding also yields
	$\tr B^2 = 2p^2 - 9p + 10$, which provides another proof of \eqref{eq-BSOS}.  In fact,
	since we already have an independent proof of \eqref{eq-BSOS},
	we could work backward from \eqref{eq-Backward} to
	provide another proof of \eqref{eq-KloostermanSOS}.	

	From \eqref{eq-KloostermanOrthogonal} and \eqref{eq-KloostermanSOS} we obtain
	(for $c \neq 0,1$)
	\begin{align}
		\sum_{u=0}^{p-1} \left[ K(u) - K(cu) \right]^2 &= 2p^2, \label{eq-KSOS-} \\
		\sum_{u=0}^{p-1}\left[ K(u) + K(cu) \right]^2 &= 2p^2 - 4p . \label{eq-KSOS+}
	\end{align}
	Other such quadratic identities (e.g., \cite[eqs.~3.5, 3.8]{Lehmer}) might be deduced 
	from Table \ref{TableCharacter} using the generalized orthogonality relations \cite[Thm.~2.13]{Isaacs}:
	\begin{equation*}
		\frac{1}{ | {\bf G} |} \sum_{G \in {\bf G}} \chi_i(GH) \chi_j(G^{-1}) = \delta_{i,j} \frac{ \chi_i(H) }{ \chi_i(I)}.
	\end{equation*}
	
	Applying \cite[Thm.~30.4]{JL} (see also \cite[Prob.~3.9]{Isaacs}) we obtain the formula
	\begin{equation*}
	c_{i,j,k} = \frac{ | {\bf G} | }{ | C_{ \bf G}(G_i) | | C_{ \bf G}(G_j) | } \sum_{u=1}^{2p} \frac{ \chi_u(G_i) \chi_u(G_j) \overline{\chi_u(G_k)} }{ \chi_u(I)},
	\end{equation*}
	which for $i = 1$ and $1 \leq j,k \leq p-1$ yields
	an identity equivalent to \cite[Prop.~6.2]{Ku75}:
	\begin{equation*}
		\sum_{u=0}^{p-1} K(u)K(ju) K(ku) = \left( \frac{\beta(1,j,k)  }{p} \right) p^2 + 2p
	\end{equation*}
	(this can also be easily deduced by computing the $(j,k)$ entry of $T = UDU$).
	By \eqref{eq-Beta} we have $\beta(1,1,1) = -3$ from which we obtain \cite[eq.~1]{LehmerCubes},
	\cite[eq.~3.22]{Lehmer}, and \cite[eq.~(70)]{Salie}:
	\begin{equation}\label{eq-KloostermanCubic}
		\boxed{
		\sum_{u=0}^{p-1} K^3(u) = 
		\begin{cases}
			p^2 + 2p  & \text{if $p \equiv 1 \pmod{3}$}, \\
			-p^2 + 2p  & \text{if $p \equiv 2 \pmod{3}$}.
		\end{cases}
		}
	\end{equation}
	
\subsection{Quartic formulas and Kloosterman's bound}
	Computing the $(j,j)$ entry of $T^2 = UD^2U$, we obtain
	\begin{equation*}
		\underbrace{ \frac{1}{p^2} \left( \sum_{u=0}^{p-1} K(u)^2 K(ju)^2 +1 + f^3 \right) }_{\text{$(j,j)$ entry of $UD^2U$}}
		= 
		\underbrace{
		\begin{cases}
			1 + 2(p-3) + f & \text{if $j = 1$},\\
			2 +2(p-5) +2 & \text{if $j \neq 1$ and $(\frac{j}{p})=1$},\\
			2(p-3) +2 & \text{if $j \neq 1$ and $(\frac{j}{p})=-1$},\\
		\end{cases}
		}_{\text{$(j,j)$ entry of $T^2$ obtained from Table \ref{TableRows}}}
	\end{equation*}
	leading us to \cite[Prop.~6.3]{Ku75} and \cite[eq.~3.18]{Lehmer}:
	\begin{equation}\label{eq-KloostermanQuartic}
		\boxed{
		\sum_{u=0}^{p-1} K^2(u) K^2(ju)
		= 
		\begin{cases}
			2p^3 - 3 p ^2 - 3p  & \text{if $j=1$},\\
			p^3 - 3 p^2 - 3p  & \text{if $j \neq 1$ and $(\frac{j}{p})=1$},\\
			p^3 - p^2 - 3p  & \text{if $j \neq 1$ and $(\frac{j}{p})=-1$}.
		\end{cases}
		}
	\end{equation}
	Based upon this we obtain the following result of Kloosterman himself \cite{Kloosterman}.
	
	\begin{theorem}[Kloosterman]\label{TheoremKloosterman}
		$|K(u)| < 2^{\frac{1}{4}} p^{\frac{3}{4}}$ for all $u$.
	\end{theorem}
	
	\begin{proof}
		Let $j = 1$ in \eqref{eq-KloostermanQuartic}, observe that $|K(u)|^4 < 2p^3$, then take fourth roots.
	\end{proof}
	
	We remark that the simple proof above achieves a better constant 
	(namely $2^{\frac{1}{4}}$ in place of $3^{\frac{1}{4}}$ -- see also \cite[eq.~72]{Salie}) than the recent proof in \cite{DRHB}.  
	
	Although it is not clear whether one can obtain the Weil bound \eqref{eq-Weil} using these methods, we
	have at least demonstrated that the unitary similarity $A= UDU^*$ encodes enough information
	about Kloosterman sums to obtain nontrivial results.  Furthermore, we can establish that the exponent $\frac{1}{2}$
	appearing in the Weil bound cannot be improved.  Indeed, from 
	\eqref{eq-KloostermanSOS} and \eqref{eq-KloostermanQuartic} we have
	\begin{align*}
		2p^3 - 3p^2 - 3p
		&= \sum_{u=0}^{p-1} K(u)^4 
		\leq \max\{ K(u)^2 \} \sum_{u=0}^{p-1} K^2(u) \\
		&\leq \max\{ K(u)^2\} (p^2-p)
	\end{align*}
	whence
	\begin{equation*}
		\max\{ K(u)^2 \} \geq \frac{2p^2-3p-3}{p-1} =
		2p-2 + \frac{p-5}{p-1} > 2(p-1).
	\end{equation*}
	In other words, there exists some $u$ such that
	\begin{equation}\label{eq-KloostermanMinimum}
		|K(u)| \geq \sqrt{2}(p-1)^{\frac{1}{2}}.
	\end{equation}
	
%
%
%
	
\subsection{Symmetric functions of Kloosterman sums}
	Recall that the coefficients $c_j$ in the expansion 
	\begin{equation*}
		\det(X - \lambda I) = c_0 \lambda ^n + c_1 \lambda ^{n-1} + \cdots + c_n
	\end{equation*}
	of the characteristic polynomial of a $n\times n$ matrix $X$ are given by B\^ocher's recursion
	\begin{equation*}
	c_0 =1, \qquad c_j = -\frac{1}{j} \left[c_{j-1} \tr X + c_{j-2}\tr X^2 + \cdots +c_0 \tr X^j \right].
	\end{equation*}
	Applying this procedure to the diagonal matrix $X = \diag(K_0,K_1,\ldots,K_{p-1})$ and using 
	\eqref{eq-KloostermanTrace}, \eqref{eq-KloostermanSOS}, \eqref{eq-KloostermanCubic}, and
	\eqref{eq-KloostermanQuartic}, 
	we obtain $c_0 = 1$, $c_1 = 0$,
	\begin{align*}
		c_2 &= \tfrac{1}{2}(\tr^2 X - \tr X^2 )\\ &= -\tfrac{1}{2}(p^2-p) ,\\[5pt]
		c_3 &= - \tfrac{1}{6}\left[ (\tr X)^3 + 2 \tr X^3 - 3 (\tr X)( \tr X^2)\right] \\
		&= - \frac{p}{3}\left[ \left(\frac{-3}{p} \right)p +2\right] ,\\[5pt]
		c_4 &= \tfrac{1}{24} \left[(\tr X)^4 -6 (\tr X)^2(\tr X^2) +3 (\tr X^2)^2+8 (\tr X)(\tr X^3) -6 \tr X^4\right] \\
		&=\tfrac{1}{8}p(p-3)(p^2-3p-2).
	\end{align*}
	We therefore obtain
	\begin{equation}\label{eq-CharacteristicPolynomial}
		\boxed{
		\prod_{u=0}^{p-1} (\lambda  - K_u) 
		= \lambda ^p - \frac{1}{2}(p^2-p)\lambda ^{n-2} -  \frac{p}{3}\left[ \left(\frac{-3}{p} \right)p +2\right] \lambda ^{n-3} + \cdots,
		}
	\end{equation}
	which agrees with \cite[p.~403]{Lehmer}.  The preceding now yields formulas for certain symmetric functions
	of Kloosterman sums:
	\begin{align*}
		\sum_{0\leq j<k\leq p-1} K_jK_k &= -\frac{1}{2}(p^2-p), \\
		\sum_{0\leq i<j<k\leq p-1} K_iK_jK_k &=  \frac{p}{3}\left[ \left(\frac{-3}{p} \right)p +2\right] , \\
		\sum_{0\leq i<j<k<l \leq p-1} K_i K_j K_k K_l &= \tfrac{1}{8}p(p-3)(p^2-3p-2).
	\end{align*}

	A different approach to \eqref{eq-CharacteristicPolynomial} can be based
	upon the fact that the matrix
	\begin{equation*}
		X =\small
		\left(
		\begin{array}{cccc|cc|c}
			c_{1,1,1}&c_{1,1,2}&\cdots&c_{1,1,f}&0&0&f\\
			c_{1,2,1}&c_{1,2,2}&\cdots&c_{1,2,f}&1&1&0\\
			\vdots&\vdots&\ddots&\vdots&\vdots&\vdots&\vdots\\
			c_{1,f,1}&c_{1,f,2}&\cdots&c_{1,f,f}&1&1&0\\
			\hline
			0&1&\cdots&1&0&1&0\\
			0&1&\cdots&1&1&0&0\\
			\hline
			1&0&\cdots &0&0&0&0\\
		\end{array}
		\right)
	\end{equation*}
	is similar to the diagonal matrix $D = \diag(K_1,K_2,\ldots,K_f,-1,-1,f)$ \eqref{eq-TruncatedD}.
	Indeed, the first matrix is similar to the truncated matrix $T$ \eqref{eq-TruncatedT}, which is itself
	unitarily similar to $D$.  Using the fact that one may add a multiple of one row (resp.~column)
	to another inside a determinant, one easily obtains \cite[Lem.~14]{Ku75}:

	\begin{theorem}
		The Kloosterman sums $K_0,K_1,K_2,\ldots,K_f$ are precisely the eigenvalues of the matrix
		\begin{equation*}
				\left(\small
					\begin{array}{cccc|c}
						c_{1,1,1}-f&c_{1,1,2}-f&\cdots&c_{1,1,f}-f&-f\\
						c_{1,2,1}&c_{1,2,2}&\cdots&c_{1,2,f}&1\\
						\vdots&\vdots&\ddots&\vdots&\vdots\\
						c_{1,f,1}&c_{1,f,2}&\cdots&c_{1,f,f}&1\\
						\hline
						0&2&\cdots&2&1\\
					\end{array}
				\right),
		\end{equation*}
		where the coefficients $c_{i,j,k}$ are defined by \eqref{eq-cijk} and $f = p-1$.
	\end{theorem}
	
Before proceeding, we remark that modifications of our main construction
apply to various generalizations of classical Kloosterman sums.  For instance, one might consider
the Galois field $\F_{p^n}$ in place of $\Z/p\Z$.  Moreover, our general scheme also applies to
hyper-Kloosterman sums (the appropriate analogue of our group ${\bf G}$ is discussed in \cite[p.~16]{Ku75}).

\section{Ramanujan multigraphs}\label{SectionRamanujan}

	Certain principal submatrices of the $T_i$ can be used to construct multigraphs
	having desirable spectral properties.
	To be more specific, a \emph{multigraph} is a graph that is permitted to have multiple edges
	and loops.\footnote{The terminology in the literature
	is somewhat inconsistent.  The term \emph{multigraph} is sometimes reserved
	for graphs with multiple edges but no loops.  If loops are present, then the 
	term \emph{pseudograph} is used.}

	Associated to a multigraph $\G$ is its adjacency matrix $A(\G)$,
	the real symmetric matrix whose rows and columns are indexed by the vertices
	$v_1,v_2,\ldots,v_n$ of $\G$ and whose $(j,k)$ entry $a_{j,k}$ is the number of 
	edges connecting $v_k$ to $v_j$.  In particular, if $j = k$ then $a_{j,j}$ counts the number
	of loops attached to the vertex $v_j$.   We refer to the eigenvalues of $A(\G)$ as the
	\emph{eigenvalues of $\G$}.

	The \emph{degree} of a vertex is the number of edges terminating at that vertex.
	We say that a multigraph $\G$ is \emph{$d$-regular} if each vertex has degree $d$.
	In this case $d$ is an eigenvalue of $\G$ with corresponding eigenvector $(1,1,\ldots,1)$.
	On the other hand, $-d$ is an eigenvalue of $\G$ if and only if $\G$ is bipartite,
	in which case the multiplicity of $-d$ corresponding to the number of connected components of $\G$.
	An easy application of the Gerschgorin disk theorem indicates that
	every eigenvalue of $\G$ belongs to the interval $[-d,d]$.  We therefore
	label the eigenvalues of a $d$-regular multigraph $\G$, according to their multiplicity, as follows:
	\begin{equation*}
		d \geq \lambda_0(\G) \geq \lambda_1(\G) \geq \cdots \geq \lambda_{n-1}(\G) \geq -d.
	\end{equation*}
	The eigenvalues of $\G$ of a $d$-regular multigraph which lie in the open interval $(-d,d)$
	are called the \emph{nontrivial eigenvalues} of $\G$.  We let $\lambda(\G)$ denote the absolute
	value of the nontrivial eigenvalue of $\G$ which is largest in magnitude.

	Following \cite{MurtySurvey}, we say that 
	a \emph{Ramanujan multigraph} is a $d$-regular multigraph $\G$ satisfying
	\begin{equation*}
		\lambda(\G) \leq 2 \sqrt{d-1}.
	\end{equation*}
	For instance, the Petersen graph is an example of a $3$-regular Ramanujan graph (see Figure \ref{FigurePetersen}).
	\begin{figure}[htb!]
		\begin{center}
			\includegraphics[width=1.25in]{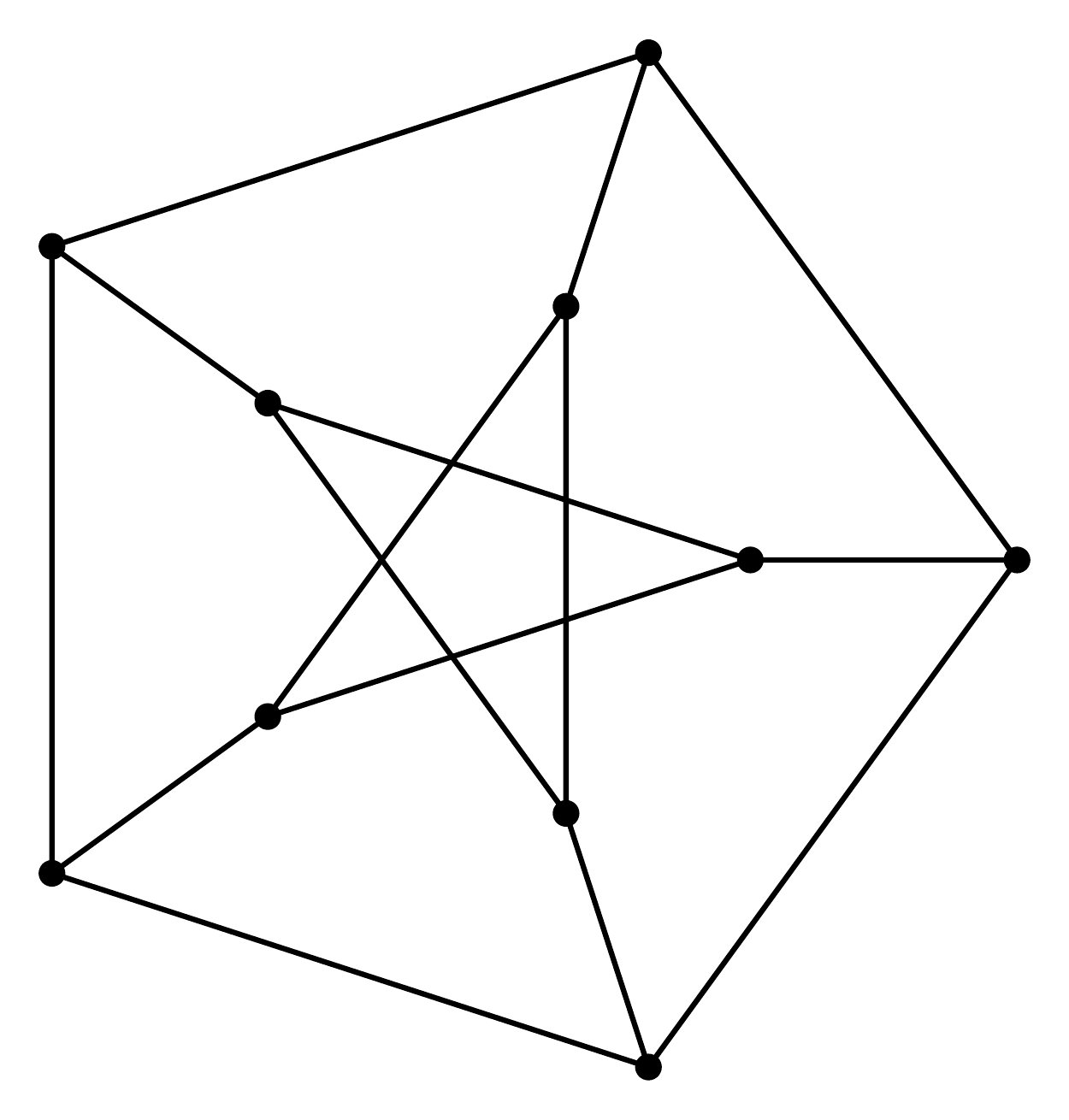}
			\caption{\footnotesize The Petersen graph is $3$-regular and has characteristic 
			polynomial $(z-3) (z+2)^4 (z-1)^5$.  The nontrivial
			eigenvalues $1$ and $-2$ are both smaller than $2\sqrt{2}$ in absolute value whence the Petersen graph is 
			Ramanujan.}
			\label{FigurePetersen}
		\end{center}
	\end{figure}
	There is a vast literature dedicated to the study of \emph{simple} (i.e., no loops or multiple edges) Ramanujan graphs.
	We refer the reader to the seminal papers \cite{LPS,Margulis,Morgenstern} and the texts
	 \cite{DSV} and \cite{LubotzkyBook} for more information.
	 
	We are now in a position to construct a family of Ramanujan multigraphs:
	 
	\begin{theorem}
		Let $p \geq 5$ be an odd prime, $1 \leq i \leq p-1$, and let $\beta(i,j,k)$ be given by \eqref{eq-Beta}.
		\begin{enumerate}\addtolength{\itemsep}{0.5\baselineskip}
			\item The multigraph $\G$ whose adjacency matrix is given by the matrix
				\begin{equation*}
					a_{j,k} = 1 + \left( \frac{\beta(i,j,k)}{p} \right)
				\end{equation*}
				is a $(p-2)$-regular Ramanujan multigraph on $p$ vertices.
			\item If $p \equiv 3 \pmod{4}$, then setting $a_{i,i} = 1$ and
				\begin{equation*}
					a_{j,k} = 1 + \left( \frac{\beta(i,j,k)}{p} \right)
				\end{equation*}
				otherwise yields a $(p-3)$-regular Ramanujan multigraph on $p-1$ vertices.
		\end{enumerate}
	\end{theorem}
	
	\begin{proof}
		The regularity of the resulting multigraphs are ensured by \eqref{eq-RowSum}
		and the general form \eqref{eq-TruncatedT} of $T_i$.  The fact that they are Ramanujan
		follows from \eqref{eq-Chain01} and \eqref{eq-Chain02}.
	\end{proof}

	For instance, letting $p=7$ and $i = 1$ we obtain the adjacency matrix 
	\begin{equation}\label{eq-Graph01}
		A=\small
		\left(
		\begin{array}{ccccccc}
			 2 & 0 & 2 & 1 & 0 & 0 & 0 \\
			 0 & 1 & 0 & 1 & 0 & 2 & 1 \\
			 2 & 0 & 0 & 2 & 0 & 0 & 1 \\
			 1 & 1 & 2 & 0 & 0 & 0 & 1 \\
			 0 & 0 & 0 & 0 & 2 & 2 & 1 \\
			 0 & 2 & 0 & 0 & 2 & 0 & 1 \\
			 0 & 1 & 1 & 1 & 1 & 1 & 0
		\end{array}
		\right)
	\end{equation}
	corresponding to the multigraph depicted in Figure \ref{FigureGraph01}.
	\begin{figure}[htb!]
		\begin{center}
			\includegraphics[width=1.25in]{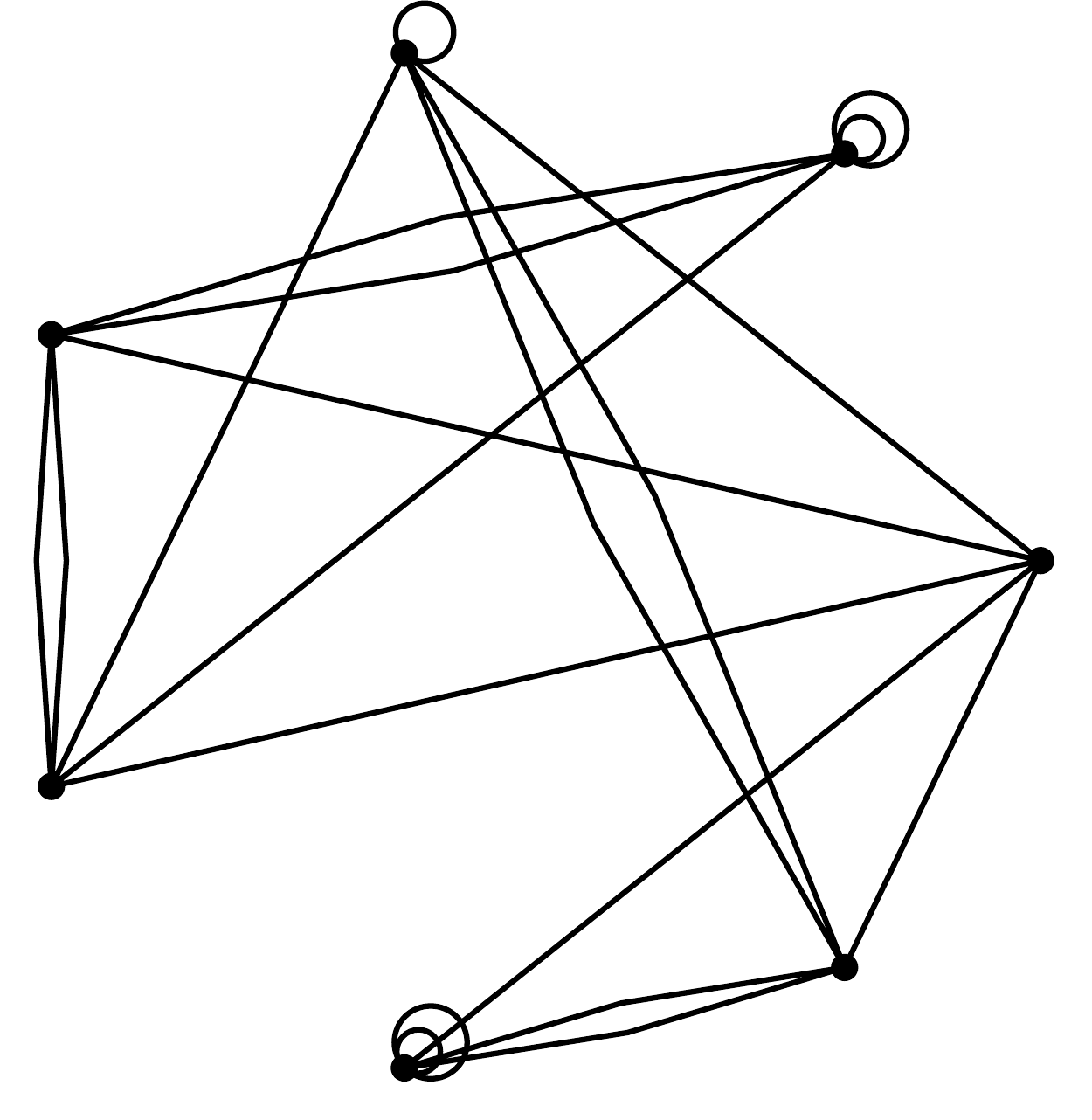}
			\caption{\footnotesize The Ramanujan multigraph corresponding to the adjacency matrix \eqref{eq-Graph01}.}
			\label{FigureGraph01}
		\end{center}
	\end{figure}

\bibliography{Kloosterman}

\end{document}